\documentclass{amsart}
\usepackage{amsmath,amssymb,amscd,graphicx,epsfig}
\usepackage{mathrsfs}
\everymath{\displaystyle}
\usepackage{mathabx}
\usepackage{dsfont} % permet de faire l'ensemble des reels, etc en tapant \mathds{R}, etc

\newtheorem{teo}{Theorem}[section]
\newtheorem{lem}[teo]{Lemma}
\newtheorem{cor}[teo]{Corollary}

\newtheorem{prop}[teo]{Proposition}
\newtheorem{lem-defi}[teo]{Lemma-Definition}

\newcommand{\mn}{\mathbb{N}}
\newcommand{\mr}{\mathbb{R}}

\newcommand{\mz}{\mathbb{Z}}
\newcommand{\mg}{\mathbb{G}}
\newcommand{\mq}{\mathbb{Q}}
\newcommand{\ms}{\mathbb{S}}

\newtheorem{theo}{Theorem}[section]

\newtheorem{lemma}[theo]{Lemma}

\newtheorem{rem}[theo]{Remark}

\newcommand{\Ff}{{\mathcal F}}

\newcommand{\I}{{\mathcal I}}

\newcommand{\Mm}{{\mathcal M}}

\newcommand{\Nn}{{\mathcal N}}

\newcommand{\Pp}{{\mathscr P}}
\newcommand{\sPp}{{\scriptscriptstyle{\mathscr P}}}

\newcommand{\Ww}{{\mathcal W}}

\newcommand{\RR}{{\mathds R}}
\newcommand{\ZZ}{{\mathds Z}}

\newcommand{\tn}{\vvvert}

\DeclareMathOperator{\Ev}{Ev}

\begin{document}

\title{Tilings of the plane\\
and\\Thurston semi-norm} 

\pagestyle{headings}
\author{Jean-Ren\'e Chazottes, Jean-Marc Gambaudo, and Fran\c cois  Gautero}

\begin{abstract} 
We show that the problem of tiling the Euclidean plane with a finite set of polygons (up to translation)
boils down to prove the existence of zeros of  a non-negative convex function defined on a finite-dimensional simplex.
This function is  a generalisation, in the framework of branched surfaces, of the Thurston semi-norm originally defined for compact $3$-manifolds.
\end{abstract}

\keywords{Wang tilings, branched surfaces, translation surfaces.}

\noindent \address{{\it J.-R. Chazottes:} Centre de Physique Th\'eorique,
CNRS-\'Ecole Polytechnique, 91128 Palaisau Cedex, France.}  \email{jeanrene@cpht.polytechnique.fr}

\noindent \address{ {\it J.-M. Gambaudo: } Universit\'e Nice Sophia Antipolis - CNRS, INLN, 1361 route des lucioles, 06560 Valbonne, France.}  \email{gambaudo@unice.fr}

\noindent \address{ {\it F. Gautero: } Universit\'e Nice  Sophia Antipolis - CNRS, Laboratoire J.A.
Dieudonn\'e, 06108 Nice Cedex 02, France.}  \email{Francois.Gautero@unice.fr }

\date{\today}

\thanks{This work is part of the projects {\it Subtile} and {\it LAM} funded 
by the French National Agency for Research (ANR)}

\maketitle
\markboth{J.-R. Chazottes, J.-M. Gambaudo, F. Gautero}{Tilings of the plane and Thurston semi-norm}

%%%%%%%%%%%%%%%%%%%%%%%%%%%%%%%%%%%%%%%%%%%%%%%%%%%%%%%%%%%%%%
%\small{\tableofcontents}

\section{Introduction} 

Consider a finite collection  $\Pp= \{p_1, \dots , p_n\}$, where for $j=1, \dots, n$, $p_j$ is a  polygon  in the
Euclidean plane $\RR^2$, indexed by $j$  and with colored edges.
These decorated polygons are called  {\em prototiles}.
In the sequel a {\it rational prototile} (resp.{ \it integral prototile}) is a prototile whose vertices have rational (resp. integer) coordinates. 
A tiling of $\RR^2$ made with $\Pp$ is a collection $(t_i)_{i\geq 0}$ of polygons called {\it tiles} indexed by a symbol $k(i)$
in $\{1, \dots, n\}$, such that:
\begin{itemize}
\item the tiles cover the plane:  $\cup_{i\geq 0} t_i = \RR^2$;
\item the tiles have  disjoint interiors: $\textup{int}(t_i) \cap \textup{int}(t_j) = \emptyset$ whenever $i\neq j$;
\item whenever two distinct tiles intersect, they do it along a common edge and the colors match;
\item for each $i\geq 0$,  $t_i$ is a translated copy of $p_{k(i)}$.
\end{itemize}
We denote by $\Omega_{\!\Pp}$ the set of all tilings made with $\Pp$. 
Clearly this set may be empty.
When $\Omega_{\!\Pp}\neq \emptyset$ the group of translation acts on $\Omega_{\!\Pp}$ as follows:
\[
\Omega_{\!\Pp}\times \RR^2\ni (T, u)\mapsto T+u \in \Omega_{\!\Pp}
\]
where $T+u = (t_i+u)_{i\geq 0}$ whenever $T = (t_i)_{i\geq 0}$.
A tiling $T\in \Omega_{\!\Pp}$ is {\it periodic} if there exist two independent  vectors $u_1$ and $u_2$ in $\RR^2$
such that $T = T+u_1 = T+u_2$.  R. Berger \cite{Berger} proved that  the problem to know whether or not
$\Omega_{\!\Pp}$ is empty ({\it i.e.}\, if one can or cannot tile the plane with $\Pp$) is not decidable.
More precisely he showed that there is no algorithm that takes as input any family of prototiles $\Pp$ and gives as output in a finite time one of the following two results: $\Pp$ can tile the plane or $\Pp$ cannot tile the plane. 
Berger also showed that this undecidability is strongly related to the fact that there exist collections of prototiles $\Pp$
which tile the plane ($\Omega_{\!\Pp} \neq \emptyset$) but which cannot tile it periodically. Before we state our main result, we need some notations and definitions.\newline

The Anderson-Putnam $CW$-complex associated with a collection $\Pp$ of prototiles is the  cell complex
$AP_{\!\Pp}$ (see \cite{AP})  \footnote{Actually the construction  given by Anderson and Putnam is made in a   particular case that suppose that $\Pp$ tiles the plane. They get  a cell complex which is smaller than the one we defined here, however the basic ideas of both constructions are the same.}  made with 2-cells, 1-cells and 0-cells constructed as follows. There is  one 2-cell for each prototile and these 2-cells are glued along their colored edges by translation \cite{ AP}. An edge $e_{k_0}$ of  $p_{k_0}$ is glued to an edge $e_{k_1}$  of $p_{k_1}$  if and only if:
\begin{itemize}
\item  they have the same color;
\item  there exists a vector $v_{k_0, k_1}$ in $\RR^2$ such that  $e_{k_1}=  e_{k_0} + v_{k_0, k_1}$.
\end{itemize}

%The gluing is made as follows: $e_{k_0} \ni x \sim y\in  e_{k_1}$ whenever $x+ v_{k_0} = y+ v_{k_1}$. 
%Notice that two edges are identified if the are parallel and have the same color but these conditions are not sufficient. Notice that there is a natural  orientation of each two cells give by the orientation of $\RR^2$ but there is no natural orientation of the 1-cells. We denote by $\pi$ the associated projection $\pi: \Omega_{\!\Pp}\to AP_{\!\Pp}$.

For $i = 1,2$, the vector space of linear combinations with real coefficients of the oriented $i$-cells is denoted by $C_i(AP_{\!\Pp},\RR)$, its elements are called $i$-chains and the coefficients are called coordinates.
For any chain $c$ in $C_i(AP_{\!\Pp},\RR)$, we denote by $\vert c\vert$  its $\ell_1$-norm. 
By convention, for each $i$-chain
$ c$ , $- c$  is the chain which corresponds to an inversion of the orientation.
Notice that there is a natural orientation of the 2-cells induced by the orientation of $\RR^2$, but there is no natural orientation of the 1-cells.  Given an arbitrary orientation on each 1-cell of $ AP_{\!\Pp}$, 
the 2-cells that contain this edge are split in two parts: the positive ones for which the orientation on the edge coincides with the one induced by the orientation of the 2-cell and the negative ones for which both orientations are different.
Notice that this splitting is independent, up to reversing, on the arbitrary choice of the orientation of the 1-cells. 
%%%%%%%%
%%%%%%%%
%%%%%%%% 
\begin{figure}[h!]
\centering
\includegraphics[width=5cm]{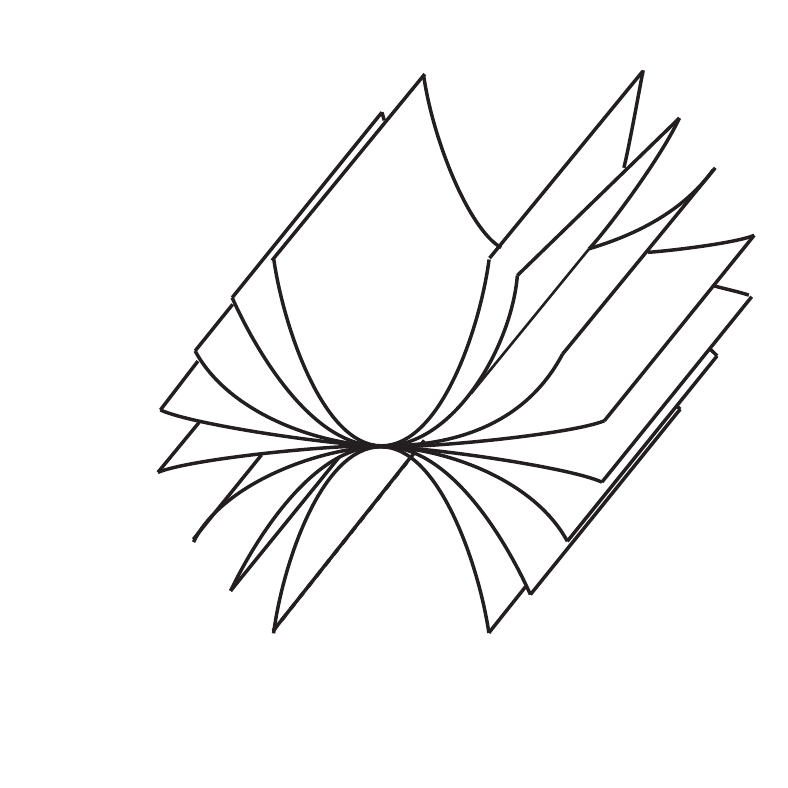}
\caption{Local view of the Anderson-Putnam complex.}
\label{APsmooth}
\end{figure}
%%%%%%%%
%%%%%%%%
%%%%%%%%

We define the linear boundary operator
\[
\partial : C_2(AP_{\!\Pp},\RR)\to C_1(AP_{\!\Pp},\RR)
\]
which assigns to any face the sum of the edges at its boundary, weighted with a positive  (resp.\,negative) sign if the induced orientation fits (resp.\,does not fit) with the orientation chosen for these edges.
The kernel of the operator $\partial$ is the  vector space of 2-cycles which we denote
$H_2(AP_{\!\Pp},\RR)$.
It is well known that (up to isomorphism) the vector space $H_2(AP_{\!\Pp},\RR)$  is a topological invariant of $AP_{\!\Pp}$  that coincides with the second singular homology group of the branched surface $AP_{\!\Pp}$, see {\em e.g.}  \cite{Spa}.
The canonical orientation of the faces allows us to characterize the vector space
$H_2(AP_{\!\Pp},\RR)\subset C_2(AP_{\!\Pp},\RR)$ as follows.  A 2-chain is a 2-cycle if and only if for each edge $e$ the sum of the coordinates of the positive faces containing $e$ is equal to the sum of the coordinates of the negative faces.
This gives a set of $m$ linear equations with integer coefficients for $n$ variables (where $m$  is the dimension of $C_1(AP_{\!\Pp},\RR)$  and $n$ the dimension of $C_2(AP_{\!\Pp},\RR)$).
These equations are called the {\it switching rules}. 
Let us say that  a 2-cycle is {\it non-negative} if its coordinates are greater than or equal to zero and
denote by $H_2^+(AP_{\!\Pp},\RR)$, the closed cone of non-negative cycles, {\it i.e.} the closed  cone of cycles with non-negative coordinates and by $S_2(AP_{\!\Pp},\RR)$ the simplex made of all non-negative cycles whose
$\ell_1$-norm is 1. 
Finally, let us say that a 2-cycle is integral (resp. rational) if its coordinates are integers (resp.\,rational numbers). 

When  $\Omega_{\!\Pp} \neq \emptyset$, {\it i.e.}\,when one can tile $\RR^2$ with $\Pp$, 
 $\Omega_{\!\Pp}$ inherits  a natural metrizable topology. A metric $\delta$ defining this topology can be chosen as
follows.  Let $B_\epsilon(0)$ stand for the open ball with radius $\epsilon$ centered at $0$ in $\RR^d$  and $B_\epsilon[T]:= \{t_j\in T: t_j\cap \bar B_\epsilon(0) \neq \emptyset\}$ be the collection of tiles in $T$ which intersect $\bar B_\epsilon(0)$. 
Consider  two tilings $T$ and $T'$ in $\Omega_{\!\Pp}$ and let $A$ be the set of $\epsilon$ in
$(0, 1)$ such that there exist $u$ and $u'$  in $\RR^d$, with $\vert u\vert, \vert u'\vert \leq \epsilon/2$, so that
$B_{1/\epsilon}[T+u] = B_{1/\epsilon}[T'+u']$. Then
\[
\delta (T, T') \, = 
\begin{cases}
\inf A & \textup{if} \,\,A \neq \emptyset\\
1 & \textup{if\, not.}
\end{cases}
\]
In words, $T$ and $T'$ are close if, up to a small translation, they agree exactly in a large neighborhood of the origin.
Equipped with such a metric, $\Omega_{\!\Pp}$ is a compact metric space and $\RR^2$ acts by translation in
a continuous way. 
Let $\Mm(\Omega_{\!\Pp})$ be the set of finite translation-invariant measures on $\Omega_{\!\Pp}$. It is non-empty
because the group $\RR^2$ is amenable. The subset of translation-invariant probability measures
is denoted by $\Theta(\Omega_{\!\Pp})$.
There exists a natural morphism:
\[
\Ev: \Mm(\Omega_{\!\Pp})\to  C_2(AP_{\!\Pp},\RR)
\]
defined by:
\[
\Ev(\mu) \, =\, \sum_{j=1}^{j=n} \frac{\mu(\pi^{-1}(p_j))}{\lambda(p_j)}\ p_j
\]
where $\lambda$ stands for the Lebesgue measure in $\RR^2$ and  $\pi:\Omega_{\!\Pp}\to AP_{\!\Pp}$ is the natural projection which associates to each tiling the location of the origin of $\RR^2$ in the (translated copy of the) prototile it belongs to.
The switching rules are related to translation invariance: 
\begin{prop}\cite {BG}
\[
\Ev(\Mm(\Omega_{\!\Pp})) \subset  H_2^+(AP_{\!\Pp},\RR)\quad\textup{and}\quad\Ev(\Theta(\Omega_{\!\Pp})) \subset  S_2(AP_{\!\Pp},\RR).
\]
\end{prop}
It turns out that the previous inclusions may not be onto: 
one can find a set  of prototiles and  non-negative coefficients  associated with these prototiles satisfying the switching rules which are not the weights of some finite translation-invariant measure (see \cite{CGHU}).
Clearly the set $\Ev(\Theta(\Omega_{\!\Pp}))$ is a closed convex subset of $S_2(AP_{\!\Pp},\RR), $ and the aim of this paper is to characterize this subset, that is to say to give geometric conditions on the non-negative  2-cycles  ensuring that they are images of finite measures. 
We shall construct  a non-negative continuous convex function $S_2(AP_{\!\Pp},\RR) \ni c \mapsto  \tn c\tn  \in \RR$
which we call the {\it asymptotic Thurston semi-norm}. Our main result is the following theorem:
%%%%%
%%%%%
\begin{theo} \label{main} 
\leavevmode\\
A collection of prototiles $\Pp$ tiles the plane if and only if  $H_2^+(AP_{\!\Pp},\RR) \neq \{0\}$
and the asymptotic Thurston semi-norm has a zero in   $S_2(AP_{\!\Pp},\RR).$ In this case
\[
c\in\Ev(\Theta(\Omega_{\!\Pp}))\Longleftrightarrow  \tn c \tn =0.
\]
\end{theo}
%%%%%%
%%%%%%

\section{The asymptotic Thurston semi-norm}

\noindent $\bullet$  {\bf The Thurston semi-norm:}  We start with a brief historical account on the Thurston semi-norm, in order to shed some light on the ideas at the origin of this work.
If $M$ is a closed ({\em i.e.} compact, without boundary)
$3$-manifold, any closed surface $\Ff$ embedded in $M$ defines an integer $2$-cycle in $C_2(M;\mz)$, and thus
a homology class $[\Ff] \in H_2(M;\mz)$, the second homology group of $ M$. Relying upon this observation,
W.P. Thurston introduced in \cite{T} a  semi-norm
$\|\cdot\| \colon H_2(M;\mz) \rightarrow \mn$
defined  by:
$$
\|u\| = \mathrm{ min } \{|\chi(\Ff)| \mathrm{ ; } [\Ff] = u\},
$$ where $\chi (.)$ stands for the Euler characteristic\footnote{We recall that the Euler characteristic of a surface is the sum of the Euler characteristics of its connected components.}.
This semi-norm vanishes on the classes which can be represented by tori, and only on these classes. Thurston (semi)-norm has been intensively studied,
see for instance \cite{McMullen, Oertel, Fried} among many others. D. Gabai \cite{Gabai} (see also \cite{Pearson}) proved that this norm does not change if one considers singular
surfaces ({\em i.e.} just ``mapped in'' surfaces, instead of embedded ones) in $M$. On another hand, U. Oertel in \cite{Oertel} used {\em branched surfaces} to describe the unit-ball of this norm: non-negative integer solutions to the so-called ``switch-equations'' define integer
$2$-cycles which in turn give surfaces embedded
in the $3$-manifold $M$, and carried by the branched surface $W$ (non-negative real solutions yield measured laminations).
When $W$ {\em is} the ambient space, i.e.\,is no more assumed to be embedded in a $3$-manifold, an integer $2$-cycle of a branched surface $W$ still defines a surface mapped in $W$, and carried by $W$ provided that the cycle is non-negative.  It is thus natural to try to extend the definition of the Thurston norm when dealing with ``abstract'' branched surfaces. Unfortunately, it happens that the multiplicativity fails
in this setting. This leads us to adapt the definition by considering an asymptotic version
$\displaystyle \tn u\tn = \lim_{n \to +\infty} \frac{\| nu\|}{n}$ (compare with \cite{Calegari}).
%$\displaystyle |||u||| = \lim_{n \to +\infty} \frac{||nu||}{n}$ (compare \cite{Calegari}).
%Investigating what is the vanishing set of this asymptotic Thurston norm in the case of a branched surface associated
%with an euclidean tiling, our main result shows that it  consists in the classes coming from the invariant measures on the associated tiling space (and not only %measures supported by tori).

\noindent $\bullet$  {\bf The asymptotic Thurston semi-norm:}

Consider a topological oriented  surface $\Ff$ made with finitely many polygons of the Euclidean plane $t_1,  \dots, t_s $   glued along their edges, full edge to full edge.   We say that $\Ff$ is a cover of $AP_{\!\Pp}$ if there exists a continuous cellular map 
  $\pi: \Ff \to AP_{\!\Pp}$ such that:
  \begin{itemize}
  \item for each $i\in \{1, \dots s\}$, there exists $k(i) \in \{1, \dots , n\}$ such that  $\pi$  is an isometry between
  $\textup{int}(t_i)$ and  $\textup{int}(p_{k(i)})$;
\item  $\pi$ is a local embedding. 
\end{itemize}
The orientation of  $\Ff$ induces an orientation of each polygon $t_i$ in $\Ff$ which may fit or not with the canonical orientation of the polygon $p_{k(i)}$. We set 
\[
\epsilon_i  \, =
\begin{cases}
1 & \textup{if the orientation is preserved}\\
-1 & \textup{if not}
\end{cases}
\]

and define for each $j\in \{1, \dots, n\}$:
$$c_j = \sum_{i\, \vert \, k(i) =j} \epsilon_j.$$
Clearly the vector $c(\Ff) \equiv  \sum_{j=1}^{j=n} c_j\cdot p_j$, is a cycle in $ H_2(AP_{\!\Pp}; \mz)$.

 For each cycle  $c \neq 0$ in $ H_2(AP_{\!\Pp}; \mz)$, we denote by $ \hat c$ the set of  all compact oriented  surfaces $\Ff$ that are covers of the flat branched surface  $AP_{\!\Pp}$ and whose associated 2-cycles $c_{\Ff}$  is equal to $c$.
\begin{lemma} \label{bla} For any cycle $c = \sum_{j=1}^{j=n} c_j\cdot p_j \neq 0$ in $ H_2(AP_{\!\Pp}; \mz)$ ,   there exists a surface $\Ff$ in $\hat c$ made with exactly $\vert c_j\vert$  copies of $p_j$, for each $j = 1, \dots, n$.
\end{lemma}
\begin{proof}  A proof of this Lemma is already given in \cite{G} for non-negative cycles, and the same arguments work here.  For each $j\in \{1, \dots, n\}$ consider $\vert c_j\vert$ copies of $p_j$ with orientation given by the sign of $c_j$. Let $p_{k_0}$ and $p_{k_1}$ be two distinct polygons chosen in the collection of $\sum_{j=1}^{j=n} \vert c_j\vert$ polygons. 

An edge $e_{k_0}$ of  $p_{k_0}$ can be glued to an edge $e_{k_1}$  of $p_{k_1}$   if and only if:
\begin{itemize}
\item  they have the same color;
\item they are translated copies one of the other; 
\item their orientations  (induced by the orientations of the polygons) are different.
\end{itemize}
If both polygons have  the same orientation, the gluing is made by the translation that maps  $e_{k_0}$ to $e_{k_1}$. If the two polygons have different orientations, the gluing is made by post-composing the translation that maps  $e_{k_0}$ to $e_{k_1}$ with the orthogonal symmetry with axis $e_{k_1}$.
We perform this series of gluing as long as possible until exhaustion of the polygons in the collection. We have thus built a compact oriented (not necessarily connected) surface $\Ff$ such that $c(\Ff) = c$. \end{proof}

For an integral cycle different from $0$, let
\[
\Vert c\Vert\,\equiv\, - \max _{\Ff\in \hat c}(\chi (\Ff))
\]
and set
\[
\Vert 0\Vert\,\equiv\, 0.
\]
%%%%%%%
%%%%%%%
\begin{rem} \label{tore}For any integral cycle $c$ different from $0$:
\[
\Vert c\Vert =0\Longleftrightarrow \hat c \;\; \textup{contains a torus.}
\]
\end{rem}
 For any  integer $n>0$ and any pair of  integral cycles $c,c^\prime$, we have:
\[
\Vert n c\Vert\leq n\Vert c\Vert\quad\textup{and} \quad \Vert c+c^\prime\Vert \leq  \Vert c\Vert+\Vert c^\prime\Vert.
\]

It follows that for any  integral cycle $c$,  the limit  of the sequence  $((1/n)\Vert n c\Vert)_n$ exists.  Hence
we define:
\[
\tn c\tn \,\equiv\, \lim_{n\to +\infty}\frac{1}{n}\Vert n c\Vert.
\]
We observe that for each integer $n>0$, and each pair of integral cycles $c,c^\prime$, we have: 
\[
\tn n c\tn =n \tn c\tn \quad \textup{and} \quad \tn c+c^\prime\tn \leq  \tn c\tn+ \tn c^\prime\tn.
\]
Thus, we can  extend the definition of the asymptotic Thurston semi-norm $\tn \cdot \tn$ to  
the cycles in $H_2(AP_{\!\Pp}; \mq)$ by setting:
\[
\tn c\tn \equiv \frac{1}{n}\tn n c\tn
\]
for all integer $n$ such that $n c$ is an integral cycle.  

\begin{rem}
For any cycle $c \in H_2(AP_{\!\Pp}; \mq),$ $\tn c\tn =0$  does not imply that $\hat c$ contains a torus.
\end{rem}

%%%%%%
%%%%%%%
\begin{lemma} There exists a constant $C_{\!\Pp}>0$ such that for any cycle in  $H_2(AP_{\!\Pp},\mq)$,  we have
\[
 \tn c\tn \leq C_{\!\Pp}  \vert c\vert.
\]
\end{lemma}
%%%%
%%%%
\begin{proof}
Let $c' = n  c$ be a cycle in  $H_2(AP_{\!\Pp},\mz)$ different from $0$, that reads $c' = \sum_{j=1}^{j=n} c_i\cdot p_i$. Following Lemma \ref{bla}, one can find a surface $\Ff$ in $\hat c$ made with exactly $\vert c_j\vert$  copies of $p_j$, for each $j = 1, \dots, n$. Let $s_{p_j}$ be the  number of vertices of the polygons in $\Pp$ and set
$s_{\!\sPp} = \max_{j=1, \dots, n} s_{p_j}$. We clearly have:
\[
\vert \chi (\Ff)\vert \leq s_{\!\sPp} \vert c'\vert
\]
whence
\[
\Vert c'\Vert  \leq s_{\!\sPp} \vert c\vert.
\]
It follows that:
\[
\tn c\tn \leq s_{\!\sPp} \vert c\vert.
\]

\end{proof}
%%%%%%%%%%%%

\begin{cor}\label{conti}
The asymptotic Thurston semi-norm  $\tn\cdot\tn$ is Lipschitz continuous  on  $H_2(AP_{\!\Pp}; \mq)$.
\end{cor}
%%%%%%%%
%%%%%%%%%

It follows  that the function $c\mapsto \tn c\tn$ can be extended to a non-negative continuous function defined on $H_2(AP_{\!\Pp}; \RR)$, and thus in particular on the whole simplex $S_2(AP_{\!\Pp},\RR)$.
We call it the {\it asymptotic Thurston map}.
From the subadditivity of the map $c\mapsto \tn c\tn$ we easily get:
\begin{lemma} \label{continue} When restricted to $S_2(AP_{\!\Pp},\RR)$, the asymptotic Thurston map is a bounded convex continuous non-negative map.
\end{lemma}

\begin{cor} When restricted to $S_2(AP_{\!\Pp},\RR)$, the set of zeros of the asymptotic Thurston map is a  (possibly empty) convex  subset of the simplex $S_2(AP_{\!\Pp},\RR)$.
\end{cor}

\section{Wang tilings }

We start with some basic definitions. 
Let $\Ww= \{w_1, \dots , w_n\}$ be a finite collection of unit squares whose vertices have integer coordinates in $\RR^2$ and colored edges. We say that $\Ww$ is a collection of  {\it Wang prototiles}. 
A Wang tiling is a tiling made with $\Ww$ such that  abutting edges of adjacent tiles have the same color.
In 1966, R. Berger \cite {Berger} gave a first example of a set of Wang prototiles which can tile the plane but cannot tile it periodically.  This example was made with a collection of 20426  Wang prototiles. Since then,  similar examples with a smaller set of Wang prototiles have been found. The state of the art is the example found by K. Culik \cite{culik}  (see also \cite{Kari}) made with 13 Wang prototiles and shown  on  Figure \ref{CulikFrancois}. 
%%%%%%%%
%%%%%%%%
\begin{figure}[h!]
\centering
\includegraphics[width=5cm]{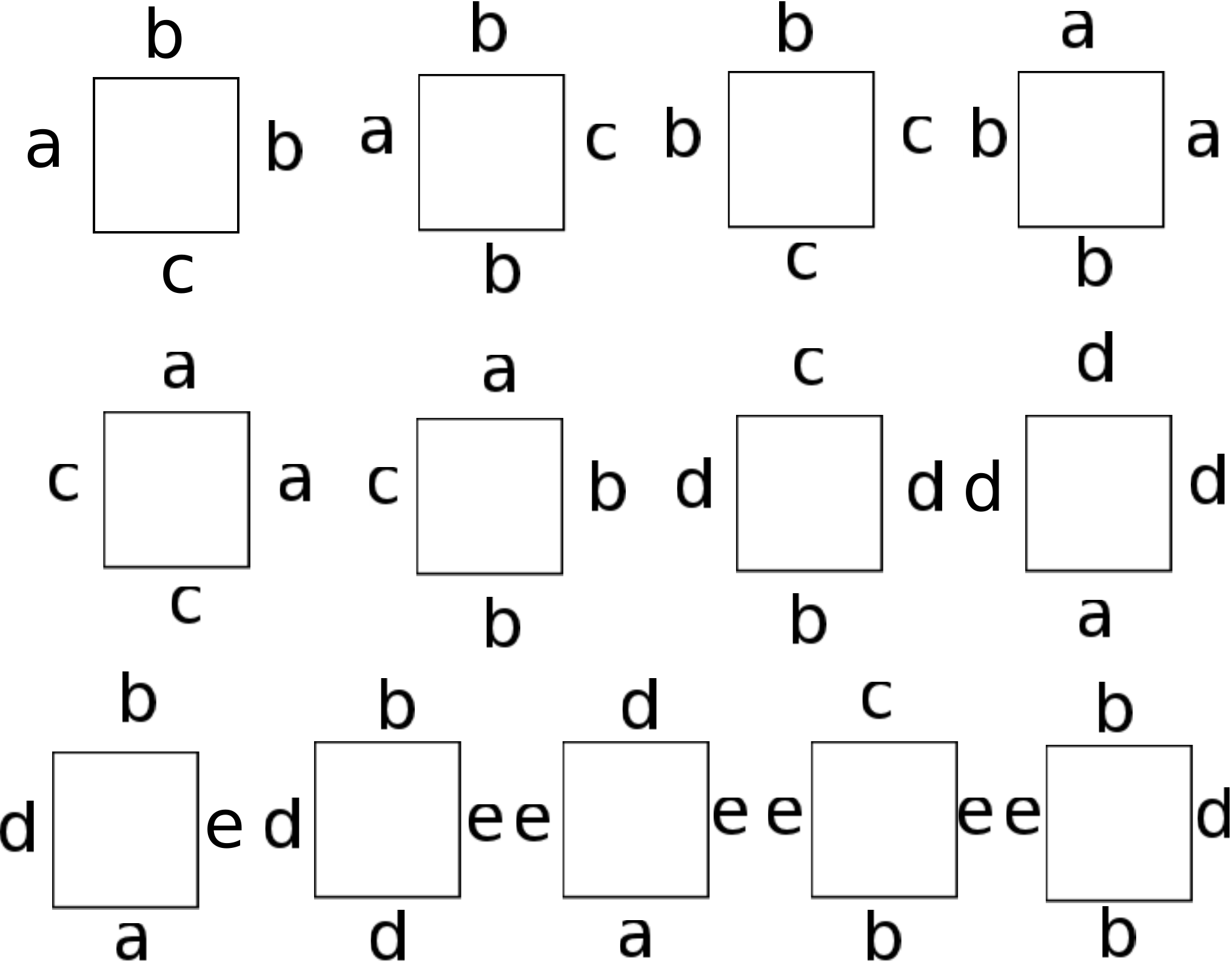}
\caption{A collection of Wang prototiles.}
\label{CulikFrancois}
\end{figure}
%%%%%%%%
%%%%%%%%

Let us introduce two notions concerning Wang tilings that we will use in the proof of the main theorem.

\noindent {\bf Forgetting colors: } 

\noindent For a given collection of Wang prototiles $\Ww =\{w_1, \dots, w_n\}$, we consider the set 
$\widehat \Ww =\{\widehat w_1, \dots,\widehat  w_n\}$, where, for $j =1, \dots, n$,  $\widehat w_j$ is deduced from $w_j$ by forgetting the colors on its sides and keeping its index $j$.  It follows that $AP_{\hat \Ww}$ is a collection of $n$ unit squares (indexed by $j$ in $\{1, \dots, n\}$) glued respectively along their horizontal edges and their vertical edges.
\begin{rem}\,
  $\Omega_{\widehat\Ww}$  is not empty and periodic orbits are dense in  $\widehat\Omega_\Ww$.
Whenever $\Omega_\Ww \neq \emptyset $, there is an isometric embedding of $\Omega_\Ww$
into $\Omega_{\widehat\Ww}$ which commutes with translations.

\end{rem}

\noindent {\bf Enforcing colors: } 

\noindent  Consider a collection of Wang prototiles  $\Ww =\{w_1, \dots, w_n\}$ and fix $p>0$. For each $j$ in $\{1, \dots, n\}$ we consider the collection of  tilings of the square $[-p-1/2, p+1/2]^2$ made  with translated copies of prototiles in $\Ww$  so that colors of common edges of adjacent tiles coincide and such that the central tile which covers $[-1/2, +1/2]^2$ is a copy of $w_j$. We denote by $\{T^1_{w_j, p}, \dots, T^{l(w_j, \Ww, p)}_{w_j, p}\}$ this collection of tilings.

For each tiling $T^l_{w_j, p}$ of $[-p-1/2, p+1/2]^2$   we consider the 4 following `colors':
\begin{itemize}
\item $\textup{Up}(T^l_{w_j, p})$ which is  the restriction of $T^l_{w_j, p}$ to $[-p-1/2, p+1/2]\times [1/2,  p+1/2]$;
\item $\textup{Down}(T^l_{w_j, p})$ which is  the restriction of $T^l_{w_j, p}$ to $[-p-1/2, p+1/2]\times [-p-1/2,-1/2 ]$;
\item $\textup{Left}(T^l_{w_j, p})$ which is  the restriction of $T^l_{w_j, p}$ to $[-p-1/2, -1/2]\times [-p-1/2, p+1/2]$;
\item $\textup{Right}(T^l_{w_j, p})$ which is  the restriction of $T^l_{w_j, p}$ to $[1/2, p+1/2]\times [-p-1/2, p+1/2]$. 
\end{itemize}
and associate  the  Wang prototile  $w_{j, l}$ whose index is the pair  $(j, l)$ and whose edges inherit the color:
\begin{itemize}
\item  $\textup{Up}(T^l_{w_j, p})$ for the top edge;
\item  $\textup{Down}(T^l_{w_j, p})$ for the bottom edge;
\item $\textup{Left}(T^l_{w_j, p})$  for the left edge;
\item and $\textup{Right}(T^l_{w_j, p})$ for the right edge. 
\end{itemize}
We denote by $\Ww^p$ the collection of Wang prototiles $w_{j, l}$ when $j$ runs from $1$ to $n$ and $l$ from $1$ to   $l(w_j, \Ww, p)$.
\begin{rem}
$\Ww$ tiles the plane if and only if,  for each $p>1,$ $\Ww^p$ (and thus  $AP_{\Ww^p}$) is well defined.
\end{rem}
     
The importance of Wang tilings  stems from the following result proved by L. Sadun and R. Williams (which is in fact valid in any dimension).
\begin{theo}\cite{SW}  \label{sw}
For any finite collection of prototiles $\Pp$ which tiles the plane, there exists a finite collection of Wang prototiles
$\Ww$ such that the dynamical systems  $(\Omega_{\!\Pp}, \RR^2)$ and 
$(\Omega_{\Ww}, \RR^2)  $ are orbit equivalent.
\end{theo}
\begin{rem} \label{rem} Actually the homeomorphism that realizes the orbit equivalence of Theorem \ref{sw}
has some important rigidity properties that  will be  detailed in the next section.
\end{rem}

\section{Reduction to  Wang tilings}

\begin{lem}\label{reduction}  Theorem \ref{main} is true if it is true for any finite collection of Wang prototiles.
\end{lem}
\begin{proof}
The proof splits in the proof of 3 claims. 

\noindent {\bf Claim 1}: 

\noindent{\it Theorem \ref{main} is true if it is true for any finite collection of  rational prototiles.}

\noindent{\it Proof of Claim 1:} In order to prove this claim we need, as announced in Remark \ref{rem}, to go deeper in the proof of Sadun and Williams of Theorem \ref{sw} in \cite{SW} and, for the sake of convenience, we sketch the construction given therein in full details. Consider a finite collection of prototiles $\Pp = \{p_1, \dots, p_n\}$ that tiles the plane. For any $\epsilon >0$, one can construct a finite collection of rational prototiles $\Pp' =  \{p'_1, \dots, p'_n\}$ such that:
\begin{itemize}
\item Each  $p'_i$ has the same number of edges as $p_i$,  and the $p'_i$'s are $\epsilon$-close to the $p_i$'s for the Haussdorff distance and the corresponding edges have the same colors;
\item The Anderson-Putnam complex $AP_{\Pp'}$ is homeomorphic to $AP_{\Pp}$. 

\end{itemize}
This amounts to solve finitely many
equations with integral coefficients and to use the fact that for such systems of equations, rational solutions are dense in the set of solutions.
The one-to-one correspondence of the 2-cells of  $AP_{\Pp'}$ and $AP_{\Pp}$ yields the natural identifications:
\[
H_2^+(AP_{\Pp'}, \RR) = H_2^+(AP_{\Pp}, \RR)\quad {\rm and}\quad S_2(AP_{\Pp'}, \RR) = S_2(AP_{\Pp}, \RR).
\] 
On the one hand,  the construction of the asymptotic Thurston norm  on $ S_2(AP_{\Pp'}, \RR)$ coincides with the  similar construction on  $S_2(AP_{\Pp}, \RR)$. On the other hand, there is a natural cone isomorphism $\I$ between $\Theta(\Omega_{\Pp'})$  and $\Theta(\Omega_{\Pp})$ which is  defined by:

\[
\I(\mu)(A) = \sum_{i=1}^{i=n} \mu(\pi^{-1}(p_i)\cap A)) \cdot \frac{\lambda (p_i)}{\lambda{(p'_i)}},
\]
for any mesurable set $A$ in $\Omega_{\!\Pp}$ and any measure $\mu$ in $\Theta(\Omega_{\Pp})$.
It follows easily that
\[
\Ev( \I(\mu)) = \Ev (\mu), \quad \forall \mu \in \Theta(\Omega_{\Pp}).
\]
This proves Claim 1.

\noindent Let us illustrate the above construction on the classical example of Penrose tilings. Consider the `thin' and `fat' triangles displayed in Figure \ref{penrosetriangles} \footnote{It is customary
to use arrowheads to indicate adjacency rules. Each triangle can be represented as a polyhedron by replacing the arrowheads by appropriate dents and bumps to fit the general definition of tilings given above.}. Together with their rotation by multiples of $2\pi/10$, they generate
a set of prototiles $\Pp$ with 40 elements  which in turn, generates the Penrose dynamical system $(\Omega_{\!\Pp}, \RR^2)$.

\begin{figure}[h!]
\centering
\includegraphics[width=8cm]{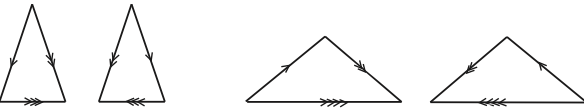}
\caption{The tiles of the Penrose tiling.}
\label{penrosetriangles}
\end{figure}

\noindent  Figure \ref{penrose} shows a patch in $\RR^2$ tiled by Penrose prototiles.

\begin{figure}[h!]
\centering
\includegraphics[width=5cm]{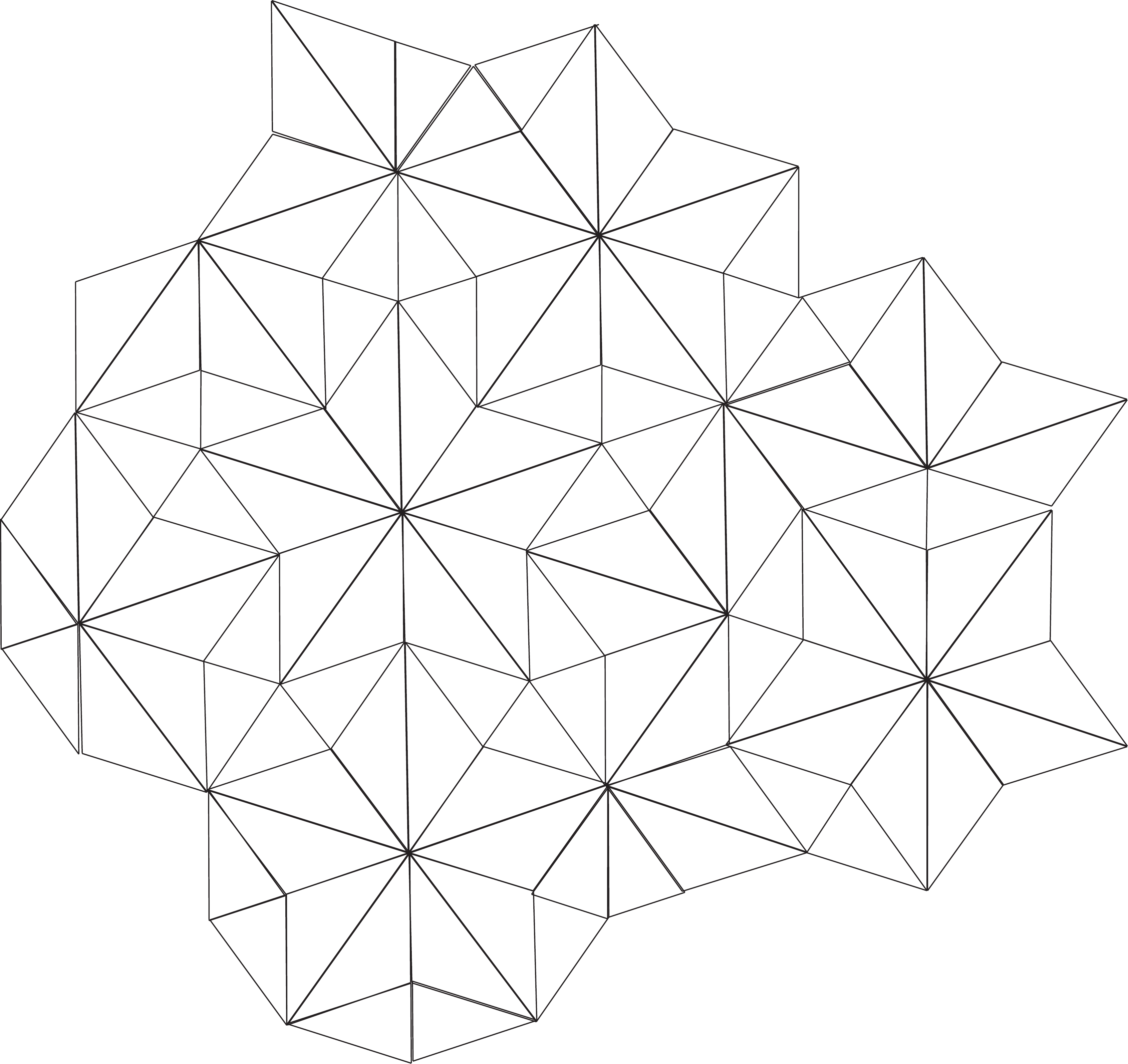}
\caption{A patch of a Penrose tiling.}
\label{penrose}
\end{figure}

\noindent Figure \ref{rationalpenrose} shows now a patch tiled with rational prototiles. 

\begin{figure}[h!]
\centering
\includegraphics[width=5cm]{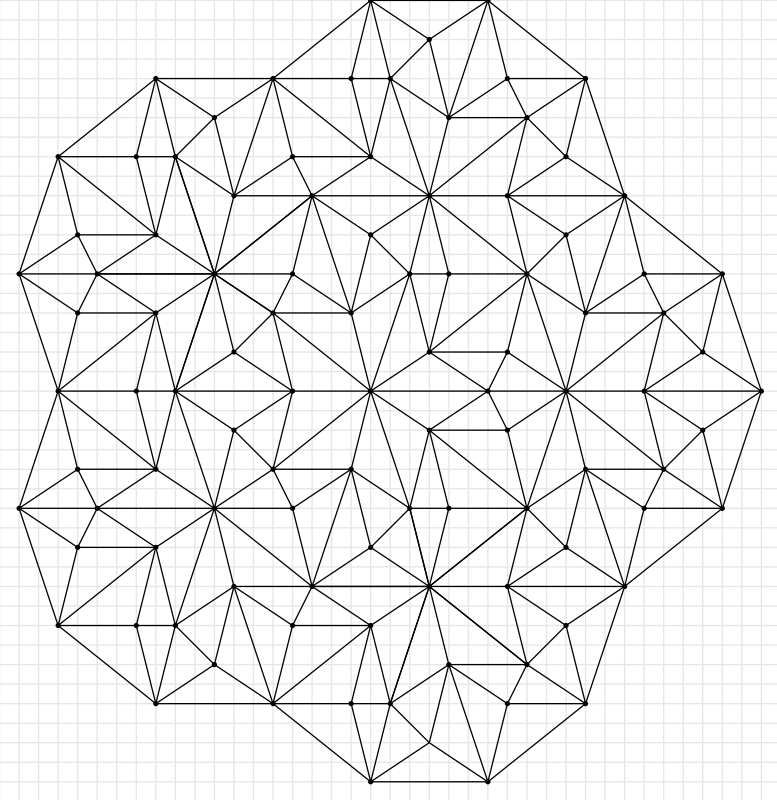}
\caption{A patch of the rational Penrose tiling.}
\label{rationalpenrose}
\end{figure}

\bigskip
\noindent {\bf Claim 2}: 

\noindent{\it Theorem \ref{main} is true for any finite collection of  rational prototiles if it is true for any  finite collection of integral prototiles.}

\noindent{\it Proof of Claim 2: }The next step in \cite{SW} is to transform a finite collection of rational prototiles $\Pp'$ into a finite collection of integral prototiles.  Using a homothety with an integral dilatation factor, one can transform the collection $\Pp'= \{p'_1, \dots, p'_n\}$ in a family of integer  prototiles $\Pp''= \{ p''_1, \dots, p''_n\}$. Clearly both dynamical systems  $(\Omega_{\Pp'}, \RR^2)$ and $\Omega_{\Pp''}, \RR^2)$ are orbit equivalent,  the homeomorphism that realizes this equivalence maps translation invariant measures onto translation invariant measures and the two Anderson-Putnam complex $AP_{\Pp'}$ and $AP_{ \Pp''}$ are homothetic. This proves Claim 2. 

\bigskip
\noindent {\bf Claim 3}: 

\noindent{\it Theorem \ref{main} is true for any finite collection of integral prototiles if it is true for any  finite collection of Wang prototiles.}

\noindent{\it Proof of Claim 3: }
 One proceeds in 2 steps:
\begin{itemize} 

\item [1]   One replaces the straight edges of the prototiles in $\Pp''$
with zig-zags, that is with sequences of unit displacements in the coordinates directions.
We denote by $\widehat \Pp =\{\widehat p_1, \dots, \widehat p_n\}$ the new collection of prototiles obtained this way.  Figure \ref{penrosecarrele} shows  how,  in the particular case of the Penrose collection of rational prototiles,  the  patch described in Figure \ref{rationalpenrose} is transformed.

\begin{figure}[h!]
\centering
\includegraphics[width=5cm]{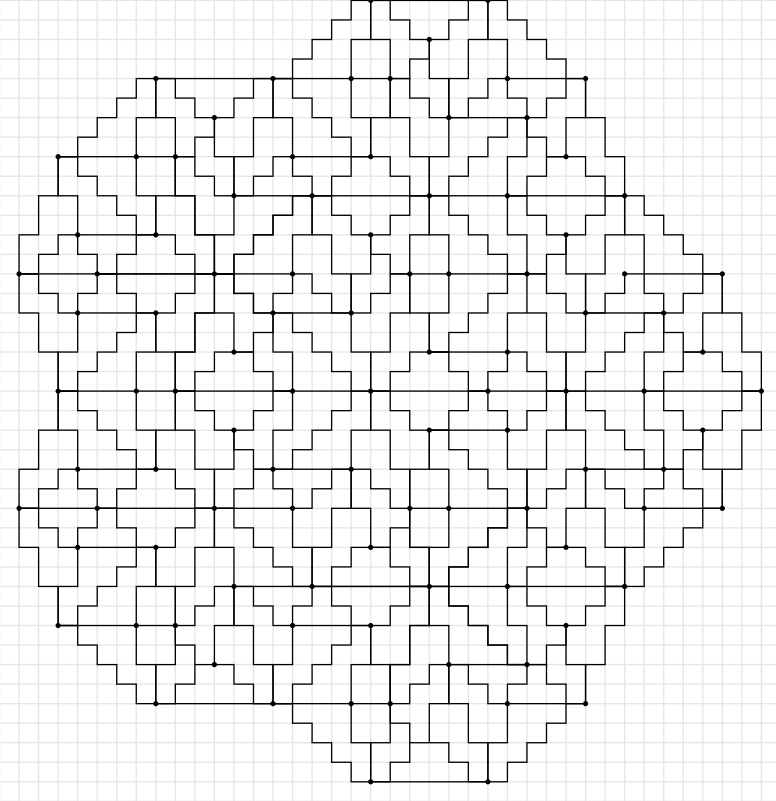}
\caption{A patch of a square Penrose tiling.}
\label{penrosecarrele}
\end{figure}

\item [2] It remains to put a label and appropriate colors on the edges of each square in  each prototile in  $\widehat \Pp$ to obtain a Wang tiling. The encoding is made in such a way that each edge of a square which is in the interior of a prototile of $\widehat\Pp$ forces as neighbors only the square which is  its neighbor in  the prototile and that  any edge  which meets the boundary of a prototile in $\widehat\Pp$ has its color given by the one of the prototile it belongs to. We denote by $\Ww$ the finite collection of Wang prototiles obtained with this construction. 
\end{itemize}
It follows that the dynamical systems $(\Omega_\Ww, \RR^2)$ and $(\Omega_{\widehat \Pp}, \RR^2)$ are
conjugate and that both Anderson-Putnam complexes $AP_\Ww$ and $AP_{\widehat \Pp}$ are homeomorphic.
This allows us to identify:
\[
H_2^+(AP_{\widehat \Pp}, \RR) = H_2^+(AP_{\Ww}, \RR)\quad \textup{and}
\quad S_2(AP_{\widehat\Pp}, \RR) = S_2(AP_{\Ww}, \RR)
\]
and proves Claim 3. 
\end{proof}

\section{Proof of Theorem \ref{main}}\label{bonanniversairetchoa}

Thanks to Lemma \ref{reduction}, we only need to prove Theorem \ref{main} 
for finite collections of Wang prototiles  $\Ww = \{w_1, \dots, w_n\}$.
  
\bigskip
\noindent $\bullet$ {\bf Assume first that $\Omega_\Ww$ is non empty.}
 
 \noindent This implies that the set of translation-invariant probability measures $\Theta(\Omega_\Ww) \neq \emptyset$. Consider an ergodic measure $\mu\in \Theta(\Omega_\Ww)$. From the Birkhoff Ergodic Theorem, we know that for $\mu$-almost every tiling $T$ in $\Omega_\Ww$ and for every prototile $w_j$ in $\Ww$
\[
\lim_{p\to +\infty} \frac{1}{{(2p+1)}^2}\ \Nn( w_j, p) = \mu(\pi^{-1}(w_j))
\]
where $\Nn( w_i, p)$ stands for the number of copies of $w_i$ that appear in $T$   in the square $[-1/2-p, p+1/2]^2$.
Fix $p>0$ and consider the periodic tiling $\widehat T_p$ in $\Omega_{\widehat\Ww}$ obtained from $T$ by repeating the pattern of $T$ in  $[-1/2-p, p+1/2]^2$. More precisely for any $(q, r)$ in $\ZZ^2$, the tile of $\widehat T_p$ centered at $(q, r)$ corresponds to the same prototile as the tile centered at $(q - u(2p+1), r- v(2p+1))$ where
$(u,v)$ are the pair of integers chosen so that $(q - u(2p+1), r- v(2p+1))\in [-1/2-p, p+1/2]^2$.  Consider the probability measure $\hat\mu_p$ which is equidistributed along the  $\mr^2$-orbit of the tiling $\widehat T_p$. 

\noindent On the one hand, notice that  $\Ev(\hat\mu_p)$ is a cycle in  $S_2(AP_{\widehat\Ww},\RR)$ which is given by
\[
\Ev(\hat\mu_p) =  \sum_{j=1}^{j=n} \hat\mu_p(\pi^{-1}(\hat w_j))\, \widehat w_j =
\frac{1}{(2p+1)^2} \sum_{j=1}^{j=n} \Nn(w_i, p) \, \widehat w_j.
\] 
It follows  that
\[
\lim_{p\to +\infty} \Ev(\hat\mu_p) = \sum_{j=1}^{j=n} \mu(\pi^{-1}(w_j))\, \widehat w_j.
\]
On the other hand, the natural inclusion $\Omega_\Ww \subset \Omega_{\widehat \Ww}$ 
allows one  to consider the measure $\mu$ as a measure $\hat \mu$ in $\Theta(\Omega_{\widehat\Ww})$ 
and  the cycle $\Ev(\mu) = \sum_{j=1}^{j=n} \mu(\pi^{-1}(w_j))\, w_j$ in $S_2(AP_{\Ww},\RR)$
can be identified  (through the above inclusion) with the cycle  
$\Ev(\hat\mu) = \sum_{j=1}^{j=n} \mu(\pi^{-1}(w_j))\widehat w_j$ in $S_2(AP_{\widehat\Ww},\RR)$.
Therefore
\[
\lim_{p\to +\infty} \Ev(\hat\mu_p) =  \Ev(\hat\mu).
\]
Since the $\mr^2$-orbit of the tiling $\widehat T_p$ is a 2-torus embedded in $\Omega_{\widehat\Ww}$, it follows directly that $\tn\Ev(\hat\mu_p)\tn = 0$. The continuity of the Thurston semi-norm (Lemma \ref{continue}) implies that $\tn \Ev(\hat\mu)\tn  = 0$ and thus $\tn \Ev(\mu)\tn  = 0$. We conclude that  $H_2^+(AP_\Ww,\RR) \neq \{0\}$  and the set of zeros of the asymptotic Thurston semi-norm on    $S_2(AP_\Ww,\RR) $ is not empty and contains $\Ev(\Theta(\Omega_\Ww)).$

\noindent $\bullet$ {\bf Assume now $H_2^+(AP_\Ww,\RR) \neq \{0\}$ 
and that the Thurston semi-norm has a zero in  $S_2(AP_\Ww,\RR) $.}
 
\noindent Let  $c\in S_2(AP_\Ww,\RR)$ be such that $\tn c\tn  = 0$. The continuity of the Thurston semi-norm and the density of rational cycles in  $S_2(AP_\Ww,\mr) $, imply that there exist a sequence of cycles $(c_\ell)_{\ell\geq 0}$ in $S_2(AP_\Ww,\mq) $ such that   $\lim_{\ell\to +\infty} c_\ell = c$,  a sequence of integers $(n_\ell)_{\ell\geq 0}$ such that
$n_\ell c_\ell$ is a non-negative integral cycle,  and a sequence of surfaces $(\Ff_\ell)_{\ell\geq 0}$ such that for each $\ell\geq 0$:
\[
\Ff_\ell\in \widehat{n_\ell  c_\ell} \quad {\rm and }\quad \lim_{\ell\to +\infty}\frac{\vert \chi(\Ff_\ell)\vert}{n_\ell} = 0.
\]
Fix now $p>0$ and, for each $\ell$ big enough, consider the surface (with boundary) $\Ff_{\ell,p}$ which is made of all the Wang tiles of $\Ff_\ell$ that are at the center of the square $[-1/2-p, p+1/2]^2$ embedded in $\Ff_\ell$. These tiles  are the one which at a distance larger than $p$ from a singular point, {\it i.e.}\, tiles which cannot be connected to a singularity with a path contained in  less than $p$  Wang tiles.  We observe that the number Wang tiles which are at a distance $0$ of a singularity is smaller than  $4\, \chi (\Ff_{l,p})\vert$. On the other hand, for any $p>0$, tiles which are at a distance $p$ from a singularity must share an edge or a vertex with tiles which are at a distance smaller than $p-1$ from the singularity and conversely, each tile which is at a distance $p-1$ from a singularity is in contact with at most eight tiles which are at a distance  $p$ from this singularity. It follows that the number of Wang tiles which are at a distance 
smaller than $p$ from a singularity is smaller than $K(p)\, \vert \chi (\Ff_{\ell,p})\vert$, where $K(p) = 4 (1+ \cdots+8^p)$. 
Let $c_{\ell, p}$ be the chain associated to $\Ff_{\ell, p}$ in $C_2(AP_\Ww,\RR)$. We have: 
\[
\vert c_{\ell, p} - n_\ell c_\ell\vert \leq K(p)\, \vert \chi (\Ff_{\ell,p})\vert
\]
and thus
\[
\lim_{\ell\to +\infty} \frac {c_{\ell, p}}{\vert c_{\ell, p}\vert} = c.
\]
%%%%%%%%%%
Clearly $\Ww^p$ is not empty since when  $\ell$ is big enough, $\Ff_{\ell, p}$ is not empty. From the very construction  of $\Ff_{\ell,p}$ we get that the chain $c_{\ell, p}$ is the image of a 2-chain  ${c^{(p)}_{\ell, p} }$ in $C_2(AP_{\Ww^p},\RR)$ through the canonical projection $\pi^{(p)}: C_2(AP_{\Ww^p},\RR)\to C_2(AP_{\Ww},\RR)$. Let $c^{(p)}$ be an accumulation point  in  $C_2(AP_{\Ww^p},\RR)$ of the sequence of normalized chains $\left(\frac {c_{\ell, p}^{(p )}}{\vert c_{\ell, p}^{(p )}\vert}\right)_{\ell>0}$. We easily check that 
\[
\vert \partial {c^{(p)}_{\ell, p} }\vert \leq K(p) \, \vert\chi (\Ff_{\ell,p})\vert
\]
which implies that $c^{(p)}$ is a non-negative  2-cycle in $H_2^+(AP_{\Ww^p},\RR)$ and  that $\pi^p(c^{(p)}) = c$. 
It follows that 
\[
c\in  \pi^{(p)} (H_2^+(AP_{\Ww^p}, \RR)), \; \forall p>0.
\]
Since it is well known (see \cite{BG} or \cite{BBG}) that 
\[
\Theta (\Omega_\Ww) = \bigcap_{p>0} H_2^+(AP_{\Ww^p}, \RR)
\] 
we deduce that $\Ww$ tiles the plane and that the zeros of the Thurston semi-norm on 
$S_2(AP_{\Ww},\RR)$  are contained in $\Theta (\Omega_\Ww) $.
  
This ends the proof of our main Theorem. 

%%%%%%%
 
\section{Discussion and examples}

Let us  examine the different situations that may occur  according to the family $\Ww$  of  Wang prototiles we consider. 
One extreme situation is when $\Ww$ does not tile the plane, in this case either $H_2^+(AP_\Ww, \RR) =\{0\}$
or   $H_2^+(AP_\Ww, \RR) \neq \{0\}$ and the Thurston semi-norm remains strictly positive in $S_2(AP_\Ww, \RR)$.
The other extreme situation is when the colors of the edges are forgotten.
In this situation, periodic orbits are dense in $\Omega_\Ww$,  $\Ev(\Theta(\Omega_\Ww))$
is the whole simplex $S_2(AP_\Ww, \RR)$ and the Thurston  semi-norm  vanishes on the whole simplex.

\noindent Let us now concentrate on the case when $H_2^+(AP_\Ww, \RR) \neq \emptyset$ and  the (convex) set of zeros of the asymptotic Thurston semi-norm (which coincides with $ \Ev(\Theta(\Omega_{\Ww}))$) is not empty. Different cases may occur:
%%%%
\begin{itemize}
\item  $ \Ev(\Theta(\Omega_{\Ww}))$ is reduced to a single cycle $c$. 
\begin{itemize}
\item If $c$  is not rational, then $\Ww$ cannot tile the plane periodically. This is exactly the situation we studied earlier for the Penrose tiling. 
\item If $c$ is rational, 
\begin{itemize}
\item either $\Vert c\Vert =0$ which means that  $\Ww$ can tile the plane periodically;
\item or  $\Vert c\Vert \neq 0$ and  $\Ww$ cannot tile the plane periodically: this is exactly what happens for the Robinson set of Wang prototiles \cite{robinsonRM}. 
\end{itemize}
\end{itemize}
\item $ \Ev(\Theta(\Omega_{\Ww}))$ is not reduced to a single cycle. In this case we are left with a series of questions, for instance:
\begin{itemize}
\item {\bf Question 1}: Can we find $\Ww$ such that rational cycles are not dense in  $ \Ev(\Theta(\Omega_{\Ww}))$?
\item {\bf Question 2}: When does   $ \Ev(\Theta(\Omega_{\Ww}))$ contain a ball in  $S_2(AP_{\Ww}, \RR)$? 
\end{itemize}

\end{itemize} 
%%%%%%
% \end{itemize}

\bigskip

\noindent {\bf Aknowledgements:} The authors acknowledge L. Sadun and Robert F. Williams for allowing the reproduction of some figures from \cite{SW}. 
They also thank an anonymous referee whose fruitful remarks led to a great  improvement of the article.
%They also thank an anonymous referee for helping them, through very frutful remarks,  to improve a lot the manuscript. 


\begin{thebibliography}{MMM}

\bibitem{AP} J. E. Anderson, I. F. Putnam, {\em Topological invariants for substitutions tilings and their $C^*$-algebras}, 
Ergodic Theory Dynam. Systems {\bf 23} (1998), No 18,  509-537. 
\bibitem {BBG}
J. Bellissard, R. Benedetti, J.-M. Gambaudo,
{\em Spaces of Tilings, Finite Telescopic Approximations and
  Gap-Labelling},  Comm. Math. Phys. {\bf 1}, (2006), 1--41. 
 
\bibitem {BG} 
R. Benedetti, J.-M. Gambaudo,
{\em On the dynamics of $\mg$-solenoids. Applications to Delone sets},
Ergodic Theory Dynam. Systems {\bf 23} (2003), No 3, 673--691. 

\bibitem{Berger}
R. Berger, 
{\em The undecidability of the domino problem}, Memoirs of the
American Mathematical Society {\bf 66} (1966), 1--72.

\bibitem{Calegari} 
D. Calegari, {\em Surface subgroups from homology}, Geometry and Topology {\bf 12} (2008), 1995--2007. 




\bibitem{culik}
K. Culik, 
{\em An aperiodic set of $13$ Wang tiles}, 
Discrete Math. {\bf 160} (1996), no. 1-3, 245--251.

\bibitem{CGHU}
J.-R. Chazottes, J.-M. Gambaudo, M. Hochman, E. Ugalde, 
{\em On the finite-dimensional marginals of shift-invariant measures}, 
Ergod. Th. Dynam. Sys. {\bf 32} (5) (2012), 1485--1500.

\bibitem{Fried}
D. Fried, {\em Fibrations over $\ms^1$ with pseudo-Anosov monodromy},
{Travaux de {T}hurston sur les surfaces, S{\'e}minaire Orsay}, {Ast\'erisque} {66}, {251--266}, {Soci\'et\'e Math\'ematique de France} (1979).

\bibitem{Gabai}
D.  Gabai, {\em Foliations and the topology of $3$-manifolds II}, {Journal of Differential Geometry} {26} (1987) {3}, {461--478}.

%\bibitem{gm}
%J.-M. Gambaudo, M. Martens, 
%Algebraic topology for minimal Cantor sets.
% Ann. Henri Poincar\'e  {\bf 7}  (2006),  no. 3, 423--446

\bibitem{G}
J.-M. Gambaudo,
{\em A note on tilings and translation surfaces}, Ergodic Theory
Dynam. Systems {\bf 1}, (2006), 179--188.  

%\bibitem{ghys} 
%E. Ghys,
%{\em  Laminations par surfaces de Riemann}, Dynamique et g\'eom\'etrie
%complexes (Lyon, 1997), Panor. Synth\`eses {\bf 8},
%Soc. Math. France, Paris, 1999 

%\bibitem{HM}
%J. Hubbard, H. Masur,
%{\em Quadratic differentials and foliations}, Acta Math. Soc. {\bf
%  66}, (1979), 221--274.  

\bibitem{Kari}
J. Kari, 
{\em A small aperiodic set of Wang tiles}, Discrete Math. {\bf 160}
(1996),  1--3. 

%\bibitem{KZ}
%A. Katok, A. Zemliakov, 
%{\em Topological transitivity of billiards in
 % polygons}, Math. Notes, {\bf 18}, (1975), 760-764.  

%\bibitem{robinsonEA}
%E. A. Robinson,
%Symbolic dynamics and tilings of $\mr^d$, in
%{\em Symbolic dynamics and its applications}, 81--119,
%Proc. Sympos. Appl. Math. {\bf 60},
%Amer. Math. Soc., Providence, RI, 2004. 

\bibitem{McMullen} 
C. McMullen,
{\em The {A}lexander polynomial of a 3-manifold and the {T}hurston
norm on cohomology}, {Annales Scientifiques de l'\'Ecole Normale Sup\'erieure} {35} {2002}, {2}, {153--171}.
              
\bibitem{Oertel}
U.  Oertel,
{\em Homology branched surfaces: {T}hurston's norm on {$H_2(M^3)$}},
{Low-dimensional topology and {K}leinian groups ({C}oventry/{D}urham, 1984)}, {London Math. Soc. Lecture Note Ser.} {112}, {253--272}, {Cambridge Univ. Press} (1986).
              
\bibitem{Pearson}
L. Person,
{\em A piecewise linear proof that the singular norm is the
{T}hurston norm}, {Topology and its Applications} {51} (1993) {3}, {269--289}.

\bibitem{robinsonRM}
R. M.  Robinson,
{\em Undecidability and non-periodicity for tilings of the plane},
Inventiones Mathematicae {\bf 12}, (1971) 177--209.  

\bibitem{robinsonjunior}
E. A. Jr. Robinson, 
{\em The dynamical properties of Penrose tilings},
Trans. Amer. Math. Soc. 348 (1996), no. 11, 4447--4464. 

%\bibitem{Sadun}
%L. Sadun,
%{\em Tiling spaces are inverse limits},
%J. Math. Phys. {\bf 44} (2003), No. 11, 5410--5414.

\bibitem{SW}
L. Sadun, R. F. Williams,
{\em Tilings spaces are Cantor fiber bundles},
Ergodic Theory Dynam. Systems  {\bf 23}  (2003), 307--316. 

\bibitem{schmidt}
K. Schmidt, 
Multi-dimensional symbolic dynamical systems,
in {\em Codes, systems, and graphical models} (Minneapolis, MN, 1999), 67--82,
IMA Vol. Math. Appl., 123, Springer, New York, 2001. 

\bibitem{Spa}
E. H. Spanier,
{\em Algebraic Topology,} McGraw-Hill, New-York, 1966.

\bibitem{T}
W. P. Thurston,
{\em A norm for the homology of {$3$}-manifolds}, Mem. Amer. Math. Soc.,
{\bf 59}, (1986) 99--130.
    



%\bibitem{Williams1}
%R. F.  Williams, 
%{\em One-dimensional non wandering sets}, Topology {\bf 6} (1967),
%473--487. 

%\bibitem{Williams2}
%R. F.  Williams,
%{\em Expanding attractors}, Inst. Hautes \'Etudes
%Sci. Publ. Math. {\bf 43} (1974), 169--203.  

\end{thebibliography}
\end{document}